\documentclass[12pt,reqno]{amsart}
\usepackage{amsfonts,amsmath,amssymb}
\usepackage{euscript,color}
\usepackage{graphpap}
\usepackage{tikz}
\allowdisplaybreaks

\advance\textheight by \topskip
\newtheorem{theorem}{Theorem}

\newtheorem{remark}{Remark}
\numberwithin{equation}{section}

\def\be{\begin{equation}}
\def\ee{\end{equation}}
\begin{document}
\title{Left to right maxima in Dyck Paths}

\author[A.~Blecher and A.~Knopfmacher]{Aubrey Blecher,  A.~Knopfmacher$^2{}$}\thanks{$^1{}$This material is based upon work supported by the National Research Foundation under grant number 86329 and 81021 respectively}
\address{A.~Blecher and A.~Knopfmacher\\
The John Knopfmacher Centre for Applicable Analysis and Number Theory\\
 School of Mathematics\\
University of the Witwatersrand, Private Bag 3, Wits 2050,
Johannesburg, South Africa}
\email{Aubrey.Blecher@wits.ac.za}
\email{Arnold.Knopfmacher@wits.ac.za}

\thanks{}
\date{\today}

\begin{abstract}
In a Dyck path a peak which is (weakly) higher than all the preceding peaks is called a strict (weak) left to right maximum.
We obtain explicit generating functions for both weak and strict left to right maxima in Dyck paths.
The proofs of the associated asymptotics make use of analytic techniques such as Mellin transforms, singularity analysis and formal residue calculus. 
\end{abstract}
\maketitle

\section{General introduction} 
A Dyck path is a lattice path in the first quadrant, that starts at the origin (0,0) with an up step ($u=(1,1)$) and thereafter only up and down ($d=(1,-1)$) steps are allowed under the conditions that it may not go below the $x$-axis and that it may terminate only if the end point is on the $x$-axis. A Dyck path with $n$ up steps must end at the point $(2n,0)$, see the definition in \cite{St}. Such a Dyck path is said to have length $2n$. For a detailed study of properties of Dyck paths see \cite{D2}. For further recent work on Dyck paths, see \cite{BBFGPW,BBK1,CFJ,CFJR,D1,MSTT,STT}.

Given an arbitrary Dyck path, we mean by a {\it{strict left to right maximum}}, any peak (successive pair of the form $ud$) in the Dyck path which is above all steps to its left. A {\it{weak left to right maximum}} is a peak which is greater than or equal to all peaks to its left.

A standard combinatorial problem is the accounting for the number of left to right maxima in combinatorial structures such as permutations and words over a fixed alphabet. In this paper we focus on obtaining a generating function for the number of left to right maxima in Dyck paths. This is a bivariate generating function which tracks the number of up steps by $z$ and the number of left to right maxima by $x$.
We also obtain a generating function for the total number of left to right maxima in Dyck paths with $n$ up steps.

As an introduction to the method we will use
for the construction of the first generating function above, here follows a sketch (Figure \ref{illustration}) of two Dyck paths of height $3$. The left to right maxima are marked in both cases by $A$ and $P$. $P$ also marks the first maximum height attained by the Dyck paths.  We begin at the origin with a $u$ step tracked in the generating function by $z$ which leaves us at the point $E$. This single up step is followed by a possibly empty upside-down Dyck path of maximum height $1$. In the left example in Figure \ref{illustration}, this part is indeed empty (and therefore not requiring $x$) but not in the second where the path between $E$ and $B$ is an upside-down Dyke path of height $1$ which gives rise to a left to right maximum thus requiring an $x$ tracker. Then we have another single $u$ step and we proceed recursively in this way leaving us eventually at the next left to right maximum which is point $A$ in the first example and $P$ in the second. In the first example, right of $A$ is again a possibly empty upside-down Dyck path, this time of maximum height $2$ where the non empty case is tracked again by $x$. We are referring to the path between $A$ and $B$ which is actually of height $1$. Once $P$ is reached, it is followed by the rest of the path which is conceived as a {\it{right to left portion}} of a Dyck path. In the section dealing with this, the generating function for these latter Dyck paths ending at height $r$ will be given and used, as will the generating function for Dyck paths of a fixed height $h$, which is used as indicated above for the possibly empty upside-down Dyck paths that occur sequentially before the point $P$ is attained.

\begin{figure}[h!]
\begin{center}
\begin{tikzpicture}[scale=0.5]
\draw[step=1cm,gray,thin] (0,0) grid (14,3);
\draw[step=1cm,gray,thin] (15,0) grid (29,3);
\draw [very thick] (0,0)--(2,2)--(3,1)--(5,3)--(7,1)--(8,2)--(10,0)--(12,2)--(14,0) ;
\draw [very thick] (15,0)--(16,1)--(17,0)--(20,3)--(22,1)--(23,2)--(24,1)--(26,3)--(29,0);
\draw (2.5,2.5) node[left]{$A$};\draw (4.5,2.5) node[left]{$B$};
\draw (1.5,1.5) node[left]{$E$};\draw (5,3) node[above]{$P$};
\draw (16.5,1.5) node[left]{$E$};
\draw (18.5,1.5) node[left]{$B$};\draw (20,3) node[above]{$P$};
\draw (2,0) node[below]{$2$};\draw (4,0) node[below]{$4$};\draw (6,0) node[below]{$6$};\draw (12,0) node[below]{$12$};
\draw (17,0) node[below]{$2$};
\draw (20,0) node[below]{$5$};\draw (26,0) node[below]{$11$};
\end{tikzpicture}\caption{Two Dyck paths of length $14$ and height $3$}
\end{center}\label{illustration}
\end{figure}
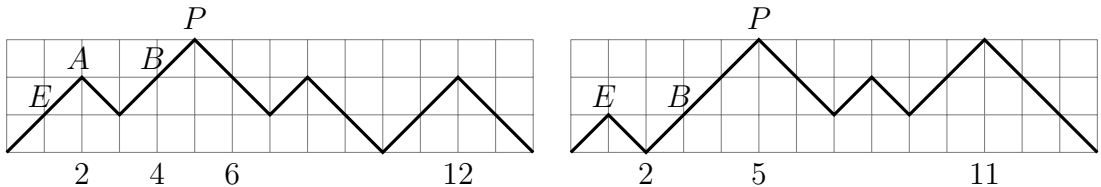

\section{Left to right maxima in Dyck paths
}

We start this section by referring to the paper \cite{H6}
by Prodinger on the first sojourn in Dyck paths. Using the notation from \cite{H6}, we let $C(h)$ be the number of  paths of height $\le h$ and steps which follow all rules of Dyck paths except that they terminate at height $h$, and we let $A(h)$ be the number of Dyck paths of height $\le h$ (which by definition end at height zero). It is shown in  \cite{H6} that

\begin{align}
C(h):=\frac{z^h\sqrt{1-4z^2}}{\lambda _1{}^{h+2}-\lambda _2{}^{h+2}}
\end{align}
and
\begin{align}
A(h):=\frac{\lambda _1{}^{h+1}-\lambda _2{}^{h+1}}{\lambda _1{}^{h+2}-\lambda _2{}^{h+2}},\label{Ah}
\end{align}

where $\lambda_1$ and $\lambda_2$, are given by

\begin{align}{\lambda_1=\frac{1+\sqrt{1-4z^2}}{2}; \lambda _2=\frac{1-\sqrt{1-4z^2}}{2}}.
\end{align}

As explained in the introductory section, we consider a sequence of possibly empty Dyck paths of height $\le h$ for $h=1,2,\hdots \ .$ At the end of each path in the sequence, we have a single up step that leads to the next left to right maximum and eventually to the first overall maximum of the entire Dyck path. We let $x$ count the number of left to right maxima attained by the Dyck path. This leads to our first theorem:

\begin{theorem}
The generating function
for the number of left to right maxima tracked by $x$, for Dyck paths of maximum height $r$ and length tracked by $z$ is
\begin{align}
F(x,z,r):=z^rxC(r)\prod_{h=1}^{r-1}(1+x(A(h)-1)).
\end{align}
\end{theorem}\label{Fxzr}

So, the total number of left to right maxima for Dyck paths of fixed height $r$ is found by differentiating the above function with respect to $x$ and setting $x=1$. The derivative at this point is given by

\begin{align}
\frac{\partial}{\partial x}F(x,z,r)\Big|_{x=1}&={z^r\text{  }C(r)\prod _{h=1}^{r-1} A(h)+z^r C(r)\prod _{h=1}^{r-1} A(h)\sum _{i=1}^{r-1} \frac{A(i)-1}{A(i)}}\notag\\
&=z^r\text{  }C(r)\prod _{h=1}^{r-1} A(h)\left(1+\sum _{i=1}^{r-1} \frac{A(i)-1}{A(i)}\right)\notag\\&=z^rC(r)\prod _{h=1}^{r-1} A(h)\left(r-\sum
_{i=1}^{r-1} \frac{1}{A(i)}\right)\label{tot}
\end{align}

Note that 
$z^r C(r)\prod_{h=1}^{r-1}A[h]$ telescopes to become

$$\frac{2^{3+2 r} z^{2 r} \left(1-4 z^2\right)}{\left(-\left(1-\sqrt{1-4 z^2}\right)^{1+r}+\left(1+\sqrt{1-4 z^2}\right)^{1+r}\right) \left(-\left(1-\sqrt{1-4
z^2}\right)^{2+r}+\left(1+\sqrt{1-4 z^2}\right)^{2+r}\right)}$$
but the full generating function becomes very complicated as a function of $z$.

To simplify this generating function, we substitute
\begin{align}
z^2=\frac{u}{(1+u)^2}
\end{align}

in (\ref{tot}) and obtain

\begin{align}
T(r):=\frac{\partial}{\partial x}F(x,z,r)\Big|_{x=1}=\frac{(1-u)^2 u^r (1+u)}{\left(1-u^{1+r}\right) \left(1-u^{2+r}\right)}\left(r-\sum _{i=1}^{r-1} \frac{1-u^{2+i}}{(1+u)
\left(1-u^{1+i}\right)}\right).
\end{align}

The full generating function for the total number of left-to-right maxima in all Dyck paths of length $n$ is
\begin{align}
Tot(u):=\sum _{r=1}^{\infty } T(r).
\end{align}

Consequently, we have the following theorem:

\begin{theorem}
The generating function $Tot(u)$ for the total number of left to right maxima in Dyck paths of length $n$ tracked by $z$ is given by
\begin{align}
Tot(u)=\sum _{r=1}^\infty \frac{(1-u)^2 u^r (1+u)}{\left(1-u^{1+r}\right) \left(1-u^{2+r}\right)}\left(r-\sum
_{i=1}^{r-1} \frac{1-u^{2+i}}{(1+u) \left(1-u^{1+i}\right)}\right),\label{totu}
\end{align}
where $z^2=\frac{u}{(1+u)^2}$.

\end{theorem}

In order to obtain the series expansion for this, we use the equivalent inverse substitution for $u$, namely

\begin{align}
u= \frac{1-2 z^2-\sqrt{1-4 z^2}}{2 z^2}\label{reverse},
\end{align}

and obtain in terms of $z$,

\begin{align}
Tot(u)&=
z^2+2 z^4+6 z^6+\boldsymbol{19 z^8}+63 z^{10}+216 z^{12}+758 z^{14}+2705 z^{16}+9777 z^{18}\notag\\&+35698 z^{20}+O[z]^{21}
\end{align}

We illustrate the bold term of the series by means of the black dots in Figure {\ref{ltrmaxillustrated}}.

\begin{figure}[h!]
\begin{center}
\begin{tikzpicture}[scale=0.5]
\draw[very thick](0,15)--(1,16)--(2,15)--(3,16)--(4,15)--(5,16)--(6,15)--(7,16)--(8,15);
\filldraw[black] (1,16.2) circle (2pt);
\draw[black] (3,16.2) circle (3pt);
\draw[black] (5,16.2) circle (3pt);
\draw[black] (7,16.2) circle (3pt);
\draw[very thick] (18,8)--(21,11)--(22,10)--(23,11)--(26,8);
\draw[black] (6,10.2) circle (3pt);
\draw[black] (13,10.2) circle (3pt);
\draw[black] (23
,11.2) circle (3pt);
\filldraw[black] (21,11.2) circle (2pt);
\draw[very thick] (0,-1)--(3,2)--(6,-1)--(7,0)--(8,-1);
\filldraw[black] (3,2.2) circle (2pt);
\draw[very thick](0,4)--(2,6)--(3,5)--(5,7)--(8,4);
\filldraw[black] (2,6.2) circle (2pt);
\filldraw[black] (5,7.2) circle (2pt);
\draw[very thick](9,4)--(12,7)--(14,5)--(15,6)--(17,4);
\filldraw[black] (12,7.2) circle (2pt);
\draw[very thick](18,4)--(19,5)--(20,4)--(23,7)--(26,4);
\filldraw[black] (19,5.2) circle (2pt);
\filldraw[black] (23,7.2) circle (2pt);
\draw[step=1cm,gray,thin] (0,15) grid (8,16);
\draw[step=1cm,gray,thin] (9,15) grid (17,17);
\draw[step=1cm,gray,thin] (18,15) grid (26,17);
\draw[step=1cm,gray,thin] (0,-1) grid (8,2);
\draw[step=1cm,gray,thin] (9,-1) grid (17,3);
\draw[step=1cm,gray,thin] (0,4) grid (8,7);
\draw[step=1cm,gray,thin] (9,4) grid (17,7);
\draw[step=1cm,gray,thin] (18,4) grid (26,7);
\draw[step=1cm,gray,thin] (0,8) grid (8,10);
\draw[step=1cm,gray,thin] (9,8) grid (17,10);
\draw[step=1cm,gray,thin] (18,8) grid (26,11);
\draw[step=1cm,gray,thin] (0,12) grid (8,14);
\draw[step=1cm,gray,thin] (9,12) grid (17,14);
\draw[step=1cm,gray,thin] (18,12) grid (26,14);
\draw[very thick](9,-1)--(13,3)--(17,-1);
\filldraw[black] (13,3.2) circle (2pt);
\draw[very thick] (9,15)--(11,17)--(13,15)--(14,16)--(15,15)--(16,16)--(17,15);
\filldraw[black] (11,17.2) circle (2pt);
\draw[very thick] (18,15)--(19,16)--(20,15)--(21,16)--(22,15)--(24,17)--(26,15);
\filldraw[black] (19,16.2) circle (2pt);
\draw[black] (21,16.2) circle (3pt);
\filldraw[black] (24,17.2) circle (2pt);
\draw[very thick](0,12)--(1,13)--(2,12)--(4,14)--(6,12)--(7,13)--(8,12);
\filldraw[black] (1,13.2) circle (2pt);
\filldraw[black] (4,14.2) circle (2pt);
\draw[very thick] (9,12)--(11,14)--(13,12)--(15,14)--(17,12);
\filldraw[black] (11,14.2) circle (2pt);
\draw[black] (15,14.2) circle (3pt);
\draw[very thick](18,12)--(20,14)--(21,13)--(22,14)--(23,13)--(24,14)--(26,12);
\filldraw[black] (20,14.2) circle (2pt);
\draw[black] (22,14.2) circle (3pt);
\draw[black] (24,14.2) circle (3pt);
\draw[very thick](0,8)--(1,9)--((2,8)--(4,10)--(5,9)--(6,10)--(8,8);
\filldraw[black] (1,9.2) circle (2pt);
\filldraw[black] (4,10.2) circle (2pt);
\draw[very thick] (9,8)--(11,10)--(12,9)--(13,10)--(15,8)--(16,9)--(17,8);
\filldraw[black] (11,10.2) circle (2pt);
\end{tikzpicture}
\caption{All 14  Dyck paths of length 8 with 19 strict left to right maxima indicated by black dots and with circles indicating the additional weak left to right maxima.}\label{ltrmaxillustrated}
\end{center}
\end{figure}
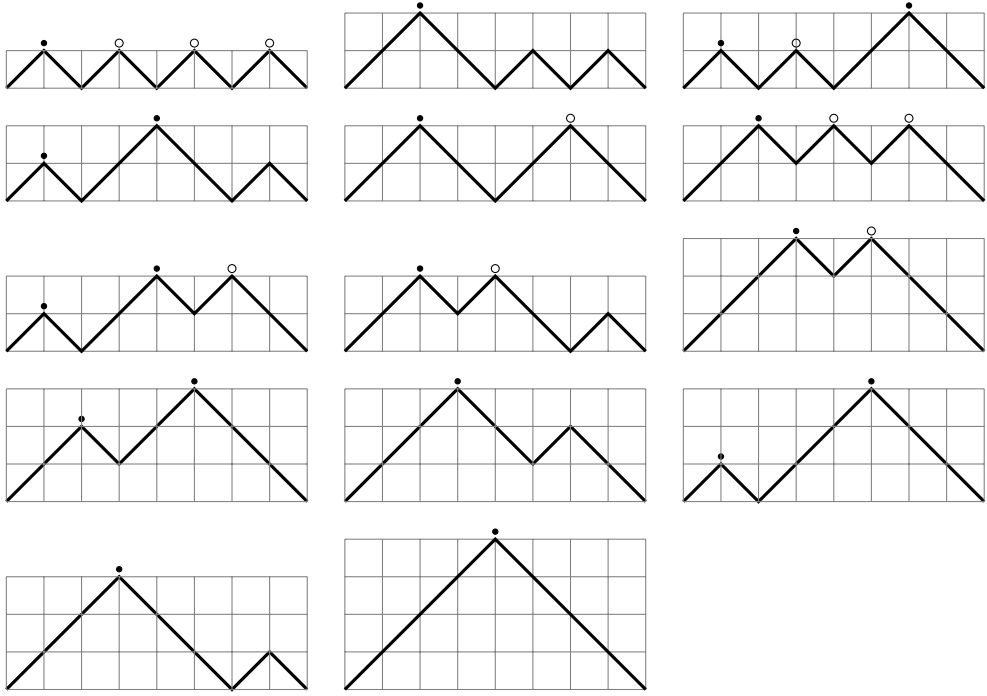

To simplify equation (\ref{totu}) we swop the order of the summations in the double sum, and thereafter use partial fractions on the second sum (which then telescopes as in line (\ref{simtotu1})) to obtain

\begin{align}
&\sum _{r=1}^{\infty } \frac{(1-u)^2 u^r (1+u)}{\left(1-u^{1+r}\right) \left(1-u^{2+r}\right)}\sum _{i=1}^{r-1} \frac{1-u^{2+i}}{(1+u)
\left(1-u^{1+i}\right)}\notag\\
&=(1-u)^2\sum _{i=1}^{\infty } \frac{1-u^{2+i}}{ \left(1-u^{1+i}\right)}\sum _{r=i+1}^{\infty } \frac{ u^r }{\left(1-u^{1+r}\right)
\left(1-u^{2+r}\right)}\notag\\
&={(1-u)^2\sum _{i=1}^{\infty } \frac{1-u^{2+i}}{ \left(1-u^{1+i}\right)}\frac{u^{1+i}}{(1-u) \left(1-u^{2+i}\right)}.\label{simtotu1}}\\\notag
\end{align}
Now changing the index of summation from $i$ to $r$,
\begin{align}
\sum _{i=1}^{\infty }\frac{1-u^{2+i}}{ \left(1-u^{1+i}\right)}\frac{u^{1+i}}{ \left(1-u^{2+i}\right)}=(1-u)\sum _{r=1}^{\infty } \frac{u^{1+r}}{
\left(1-u^{1+r}\right)}.
\end{align}

Altogether,

\begin{align}
Tot(u)&=\sum _{r=1}^\infty \frac{(1-u)^2 u^r (1+u)r}{\left(1-u^{1+r}\right) \left(1-u^{2+r}\right)}-(1-u)\sum
_{r=1}^{\infty } \frac{u^{1+r}}{ \left(1-u^{1+r}\right)}\notag\\&=\sum _{r=1}^\infty \frac{(1-u) u^r \left(r-u-r u^2+u^{3+r}\right)}{\left(1-u^{1+r}\right) \left(1-u^{2+r}\right)}\notag\\&=\sum _{r=1}^\infty \frac{ru^r-u^{1+r}-r u^{1+r}+u^{2+r}-r u^{2+r}+r u^{3+r}+u^{3+2 r}-u^{4+2 r}}{\left(1-u^{1+r}\right) \left(1-u^{2+r}\right)}.
\end{align}

Drop the first term $ru^r$ in the numerator above and  apply partial fractions to the remainder of the summand to  get

$$1-u+\frac{-1-r+2 u-r u-u^2+r u^2}{(1-u) \left(1-u^{1+r}\right)}+\frac{r+r u-r u^2}{(1-u) \left(1-u^{2+r}\right)}.$$

The separated first term with numerator $ru^r$ leads after partial fractions to

$${\frac{r u^r}{(1-u) \left(1-u^{1+r}\right)}-\frac{r u^{r+1}}{(1-u) \left(1-u^{2+r}\right)}}.$$
Altogether,
\begin{align}
Tot(u)&=\sum _{r=1}^\infty \left(1-u+\frac{-1-r+2 u-r u-u^2+r u^2}{(1-u) \left(1-u^{1+r}\right)}+\frac{r+r
u-r u^2}{(1-u) \left(1-u^{2+r}\right)}\right)\notag \\&+{\sum _{r=1}^\infty \frac{r u^r}{(1-u) \left(1-u^{1+r}\right)}-\sum _{r=1}^\infty \frac{r u^{r+1}}{(1-u) \left(1-u^{2+r}\right)}}.
\end{align}

To facilitate the evaluation of the infinite sums, we define a new function (where $\infty$ is replaced temporarily by finite $M$ in $Tot(u)$), namely:

\begin{align}
Tot2(u)&:=\sum _{r=1}^M \left(1-u+\frac{-1-r+2 u-r u-u^2+r u^2}{(1-u) \left(1-u^{1+r}\right)}+\frac{r+r
u-r u^2}{(1-u) \left(1-u^{2+r}\right)}\right)\notag \\&+{\sum _{r=1}^M \frac{r u^r}{(1-u) \left(1-u^{1+r}\right)}-\sum _{r=1}^M \frac{r u^{r+1}}{(1-u) \left(1-u^{2+r}\right)}}.
\end{align}

We now separate this into disjoint sums and shift the index of summation in the third and last sums: 

\begin{align}
Tot2(u)&=\sum_{r=1}^M (1-u)+\sum _{r=1}^M \frac{-1-r+2 u-r u-u^2+r u^2}{(1-u) \left(1-u^{1+r}\right)}+\sum
_{r=2}^{M+1} \frac{(r-1)\left(1+ u- u^2\right)}{(1-u) t(1-u^{1+r})}\notag\\&+\sum _{r=1}^M \frac{r u^r}{(1-u) \left(1-u^{1+r}\right)}-\sum _{r=2}^{M+1} \frac{(r-1) u^r}{(1-u) \left(1-u^{1+r}\right)}\notag\\
&=\sum _{r=1}^M (1-u)+\frac{-2+u}{(1-u) \left(1-u^2\right)}+\sum _{r=2}^M \frac{-1-r+2 u-r u-u^2+r
u^2}{(1-u) \left(1-u^{1+r}\right)}\notag\\&+\sum _{r=2}^M \frac{(r-1)\left(1+ u- u^2\right)}{(1-u) \left(1-u^{1+r}\right)}+\frac{M \left(1+u-u^2\right)}{(1-u) \left(1-u^{2+M}\right)}+\frac{u}{(1-u)
\left(1-u^2\right)}\notag\\&+{\sum _{r=2}^M \frac{r u^r}{(1-u) \left(1-u^{1+r}\right)}-\sum _{r=2}^M \frac{(r-1) u^r}{(1-u) \left(1-u^{1+r}\right)}-\frac{M u^{1+M}}{(1-u)
\left(1-u^{2+M}\right)}}.\label{totfinit}
\end{align}

We combine the terms in the sums from $r$ equals 2 to $M$  in (\ref{totfinit}) to get

$$\frac{-2+u+u^r}{(1-u) \left(1-u^{1+r}\right)}.$$

Then we simplify the rest to get

\begin{align}
Tot2(u)&=\sum _{r=2}^M \frac{-2+u+u^r}{(1-u) \left(1-u^{1+r}\right)}+\frac{2}{1-u^2}\notag\\&-\frac{M \left(-2+u+u^{1+M}+u^{2+M}-2
u^{3+M}+u^{4+M}\right)}{(1-u) \left(1-u^{2+M}\right)}.
\end{align}

Note that $Tot2(u)$ and $Tot(u)$ match at least for terms up to $\left[u^M\right]$.
Since for the present we are only interested in the terms up to $\left[u^M\right]$, we may set all higher power
terms equal to zero, to produce

\begin{align}
Tot2b(u)=\sum _{r=2}^M \frac{-2+u+u^r}{(1-u) \left(1-u^{1+r}\right)}+\frac{2}{1-u^2}+\frac{M (2-u)}{(1-u)
}.
\end{align}

Noting that 
$M=1+\sum _{r=2}^M 1,$

\begin{align}
Tot2b(u)&=\sum _{r=2}^M \frac{-2+u+u^r}{(1-u) \left(1-u^{1+r}\right)}-\frac{2}{1-u^2}+\frac{(2-u)}{(1-u)
}+\sum _{r=2}^M \frac{(2-u)}{(1-u) }\notag\\
&=\sum _{r=2}^M \frac{-2+u+u^r}{(1-u) \left(1-u^{1+r}\right)}+\frac{u}{1+u}+\sum _{r=2}^M \frac{(2-u)}{(1-u)
}.
\end{align}
Combine the summands in $\sum _{r=2}^M$. We may now allow $M \to \infty$ to finally obtain the simplified generating function as per the next theorem:

\begin{theorem}
The simplified generating function for the total number of left to right maxima in Dyck paths is
\begin{align}
Tot(u)=\sum _{r=1}^{\infty} \frac{(1-u) u^r}{1-u^{1+r}}.
\label{stot}
\end{align}
\end{theorem}
\subsection{Formula for total number of left-to-right maxima }
In this section, we will obtain an exact formula for the total number of left-to-right maxima in terms of a well-known arithmetic function, namely the divisor function $d(r)$. Note that $$\sum_{r=1}^\infty \frac{u^r}{1-u^r}=\sum_{r=1}^\infty d(r)\,u^r$$.

To read off coefficients from equation (\ref{stot}), we observe that for any formal power series $f(z)$
\[[z^{2n}]f(z)=[u^n](1-u)(1+u)^{2n-1}f(z(u)).\]
This can be justified by using formal residue calculus, see for example \cite{H3}. Therefore
\begin{align}
[z^{2n}]Tot(z)&=[u^n](1-u)(1+u)^{2n-1}\sum_{r=1}^\infty\frac{(1-u) u^r}{1-u^{1+r}}\notag\\
&=[u^n](1-u)(1+u)^{2n-1}\sum _{r=1}^{\infty } (d(r+1)-d(r)) u^r\notag\\
&=\sum _{r=1}^n (d(r+1)-d (r)) \left(\binom{2
   n-1}{n-r}-\binom{2 n-1}{n-r-1}\right)
.
\end{align}
Thus we have shown:
\begin{theorem}\label{Th4}
The total number of left-to-right maxima in Dyck paths of semi-length $n$ is given by
\[\sum _{r=1}^n (d(r+1)-d (r)) \left(\binom{2
   n-1}{n-r}-\binom{2 n-1}{n-r-1}\right).\]
\end{theorem}

\section{Asymptotics for strict left to right maxima}\label{as1}
In this section we find the asymptotic expression for the total number of strict left to right maxima in Dyck paths.
We will follow the approach used to study the height of planted plane trees by Prodinger in \cite{H3}. For related asymptotic calculations concerning the height of trees and lattice paths, see \cite{PP,H1,H2}.

First, we extract coefficients of $z^n$ in $Tot(u)$. That is we find

$$[z^n]\frac{1-u}{u}\sum_{r=2}^\infty \frac{u^r}{1-u^{r}}.$$

When $u$ is in terms of $z^2$,  by (\ref{reverse}) the function $Tot(u)$  has its dominant singularity at $z=1/2$ which is mapped to $u=1$. To study this further we set $u=e^{-t}$ and let $t \to 0$. Thus
\be \frac{1-u}{u}=e^t(1-e^{-t})=t+\frac{t^2}{2}+\frac{t^3}{6}+\cdots\,\,.\label{series1}\ee
To estimate the harmonic sum $f_1(t):=\sum_{r=2}^\infty \frac{e^{-rt}}{1-e^{-rt}}$ as $t \to 0$, we take the Mellin transform of $f_1(t)$, see \cite{FlajSedge09}, which is $f_1^*(s):=\int_0^\infty f_1(t) t^{s-1}\,dt.$
Thus
\[f_1^*(s)=\Gamma(s)\zeta(s)(\zeta(s)-1), \textrm{ for } Re(s) > 1.\]
By using the Mellin inversion formula, , we have
 $f_1(t)=\frac{1}{2 \pi i} \int_{2-i \infty}^{2+i \infty} f_1^*(s)\,t^{-s}\,ds$ (again see \cite{FlajSedge09}). By computing residues
this yields
\be f_1(t)\sim \frac{-1+\gamma -\log (t)}{t}+\frac{3}{4}-\frac{13 t}{144}+\cdots, \label{series2}\ee
where $\gamma$ is Euler's constant.

Let
\[g_1(t):=e^{t}(1-e^{-t})\,f_1(t).\]
 From (\ref{series1}) and (\ref{series2})
\begin{align}g_1(t)\sim  -\log (t) -1+\gamma+\left(\frac{3}{4}+\frac{1}{2}
   (-1+\gamma -\log (t))\right) t+\cdots.
\end{align}
Let $y=\sqrt{1-4z^2}$ and writing $e^{-t}=u=\frac{1-y}{1+y}$, we find $t=-\log\frac{1-y}{1+y}=2y+\frac{2y^3}{3}+\cdots$.\\
In terms of the $y$ variable, we therefore need to compute $g_1(2y+\frac{2y^3}{3}+\cdots)$.
\begin{align*}g_1\left(2y+\frac{2y^3}{3}+\cdots\right)&\sim (-1+\gamma -\log (2)-\log (y))+\frac{1}{2} (1+2 \gamma -2
   \log (2)-2 \log (y)) y\\& \qquad-\frac{y^2}{3}+\cdots.
\end{align*}

Replacing $y$ by $\sqrt{1-4z^2}$ gives
\begin{align*}
&-1+\gamma -\log (2)-\frac{1}{2} \log \left(1-4
   z^2\right)+\frac{1}{2} \left(1+2 \gamma -2 \log
   (2)-\log \left(1-4 z^2\right)\right) \sqrt{1-4 z^2}\\
&\qquad+\cdots\,\,.
\end{align*}
To use singularity analysis, see \cite{FlajSedge09}, it is convenient to put $z^2=x$, then we find the coefficient of $x^n$ in the above expression as $n \to \infty$. It is asymptotically equal to
\begin{align}
&\frac{1}{2 n}-\frac{\log (n)}{4\sqrt{\pi }
   n^{3/2}}+\frac{1-3 \gamma }{4 n^{3/2} \sqrt{\pi }} +\cdots\,.\label{B}
\end{align}

 To obtain the mean value we must divide by the total number of Dyck paths of semi-length $n$, i.e., as $n \to \infty$
\begin{equation}
\frac{1}{n+1}{2n \choose n}=2^{2n}\left(\frac{1}{n^{3/2} \sqrt{\pi}}  -\frac{9}{8\,\,n^{5/2} \sqrt{\pi}} +\frac{145}{128\,\,n^{7/2} \sqrt{\pi}}\right)+\cdots\,\,.\label{numberDyck}  \end{equation}

Hence, dividing (\ref{B}) by (\ref{numberDyck}) yields
\begin{theorem}\label{Th3}
The average number of strong left to right maxima in Dyck paths of semi-length $n$, as $n \to \infty$ is
\[\frac{\sqrt{\pi } \sqrt{n}}{2}-\frac{\log (n)}{4}+\frac{1}{4}
   (1-3 \gamma ) +O(n^{-1/2}).\]
\end{theorem}

\begin{remark} The asymptotic formula of Theorem~\ref{Th3} when $n=200$ yields 11.0257 for the average capacity. Using the exact formula of Theorem~\ref{Th4} divided by the Catalan number for $n=200$ yields 11.0503 which is indeed a very good match.
\end{remark}
\begin{remark}
The number of strong left to right maxima is bounded above
by the height of the path, which is known to be $\sqrt{\pi}\sqrt{n}$ as $n \to \infty$, (see e.g., \cite{H3}).
We see that asymptotically the average number is half of the height.
\end{remark}

\section{Weak left to right maxima in Dyck paths}
For this question we first need a generating function for Dyck paths of height $h$ which have only a single return to the $x$ axis. So using the formula above from (\ref{Ah}), we obtain the generating function for these where $h\ge1$ as 
\begin{align}
D(h,z)=z^2A(h-1).
\end{align}

Now in order to construct the generating function $E(h,x,z)$ for the number of times a Dyck path of length $n$ tracked by $z$, returns to $0$ where the latter is  tracked by a variable $x$ in the generating function, we construct a sequence of such Dyck paths where each term in the generating function for this sequence is multiplied by $x$.
Thus we obtain
\begin{align}
E(h,x,z)=\frac{1}{1-xD(h,z)}.    
\end{align}
We now reiterate the construction in Theorem \ref{Fxzr} to obtain
\begin{theorem}
The generating function
for the number of weak left to right maxima, tracked by $x$, for Dyck paths of maximum height $r$ and length tracked by $z$ is
\begin{align}
F(x,z,r):=z^{r+1}xC(r-1)\prod_{h=1}^{r}E(h,x,z).\label{Fweak}
\end{align}
\end{theorem}

To obtain the generating function for the total number of weak left to right maxima, we once again differentiate (\ref{Fweak}) and evaluate this at $x=1$. We obtain
\begin{theorem}
\label{wet}
The generating function
for the total number of weak left to right maxima for Dyck paths of length $n$ tracked by $z$ is
\begin{align}
WTot(u):=\sum _{r=1}^\infty \frac{(1-u) u^r
   \left(1-u^2\right)}{\left(1-u^{1+r}\right)
   \left(1-u^{2+r}\right)}\left(1-r+(1+u) \sum _{i=1}^r \frac{1-u^{1+i}}{1-u^{2+i}}\right)
   \end{align}
   where $z^2=\frac{u}{(1+u)^2}$.
\end{theorem}
\begin{proof}
The derivative of (\ref{Fweak}) is
$$\frac{\partial }{\partial x}F(x,z,r)\Big |_{x=1}=z^{r+1}C(r-1) \prod _{h=1}^r E(h,1,z)\left(1+
  \sum_{i=1}^r \frac{D(i,z)}{1-D(i,z)}\right).$$
Putting $z^2=\frac{u}{(1+u)^2}$ in the formula above we obtain

\begin{align}
z^{r+1} C(r-1) \prod _{h=1}^r \frac{1}{1-z^2 A(h-1)}=\frac{(1-u) u^r \left(1-u^2\right)}{\left(1-u^{1+r}\right)
\left(1-u^{2+r}\right)},
\end{align}

while 
the remaining bracketed part becomes
$$1-r+(1+u) \sum _{i=1}^r \frac{1-u^{1+i}}{1-u^{2+i}}.$$
\end{proof}

Now, we simplify Theorem \ref{wet}. The double sum becomes
$$\left(1-u^2\right)^2 \sum _{i=1}^{\infty }
   \frac{1-u^{1+i}}{1-u^{2+i}}\sum _{r=i}^{\infty } \frac{u^r}{\left(1-u^{1+r}\right)
   \left(1-u^{2+r}\right)}.$$ 
We use partial fractions on the $r$-sum and then the double sum  telescopes to
$$\frac{\left(1-u^2\right)^2}{(1-u) u}\sum_{i=1}^{\infty } \frac{u^{i+1}}{1-u^{i+2}}.$$
This is then combined with the single sum which simplifies to

\begin{align}
\sum _{r=1}^{\infty} \left(\frac{(1-u) u^r
\left(1-u^2\right) (1-r)}{\left(1-u^{1+r}\right)
\left(1-u^{2+r}\right)}+\frac{\left(1-u^2\right)^2
u^{1+r}}{(1-u) u \left(1-u^{2+r}\right)}\right).
\label{pst}
\end{align}
In order to further simplify (\ref{pst}) we replace $\infty$ by finite $M$ and then apply partial fractions to the summand of the first term which splits up as 
\begin{align}
&\frac{(-1+r) (1-u) (1+u)}{u
   \left(1-u^{1+r}\right)}+\frac{(-1+r+1) (1-u) (1+u)}{u
   \left(1-u^{2+r}\right)}-\frac{(1-u) (1+u)}{u
   \left(1-u^{2+r}\right)}.  
\end{align}

This is telescoping and simplifies to
\begin{align}
 \frac{(-1+M+1) (1-u) (1+u)}{u \left(1-u^{2+M}\right)}+\sum
   _{r=1}^M \left(\frac{\left(1-u^2\right)^2 u^{1+r}}{(1-u)
   u \left(1-u^{2+r}\right)}-\frac{(1-u) (1+u)}{u
   \left(1-u^{2+r}\right)}\right).    
\end{align}

Now replace $M$ by $\sum_{r=1}^M 1$, then reversing the previous replacement, letting $M$ tend to $\infty$, and finally combining all summands, we obtain
\begin{theorem}
The simplified generating function
for the total number of weak left to right maxima for Dyck paths of length $n$ tracked by $z$ is
\begin{align}
WTot(u)=
\sum _{r=1}^{\infty } \frac{\left(1-u^2\right)
   u^r}{1-u^{2+r}}.
\end{align}\label{fwtot}
\end{theorem}
This has series expansion
$$z^2+3 z^4+9 z^6+{\bf{29 z^8}}+98 z^{10}+341 z^{12}+1210 z^{14}+4356
   z^{16}+15860 z^{18}+58276 z^{20}+O\left(z^{21}\right).$$
This is illustrated in Figure \ref{ltrmaxillustrated}, where the dots and circles mark all 29 of the weak left to right maxima in Dyck paths of length 8.

\subsection{Formula for total number of weak left-to-right maxima }
In this section, we again obtain an exact formula for the total number of left-to-right maxima in terms of the divisor function $d(r)$. 
To read off coefficients from Theorem~\ref{fwtot}, as before
\[[z^{2n}]f(z)=[u^n](1-u)(1+u)^{2n-1}f(z(u)).\]
 Therefore
\begin{align}
[z^{2n}]WTot(z)&=[u^n](1-u)(1+u)^{2n-1}\sum_{r=1}^\infty\frac{(1-u^2) u^r}{1-u^{2+r}}\notag\\
&=[u^n](1-u)(1+u)^{2n-1}\sum _{r=1}^{\infty } (d(r+2)-d(r)) u^r\notag.
\end{align}
From this it follows that:
\begin{theorem}\label{Th4}
The total number of weak left-to-right maxima in Dyck paths of semi-length $n$ is given by
\[\sum _{r=1}^n (d(r+2)-d (r)) \left(\binom{2
   n-1}{n-r}-\binom{2 n-1}{n-r-1}\right).\]
\end{theorem}

\section{Asymptotics for weak left to right maxima}
To find an asymptotic expression for $WTot(u)$, we reiterate  the approach in Section \ref{as1}.
This yields

\begin{theorem}\label{Th3}
The average number of weak left to right maxima in Dyck paths of semi-length $n$, as $n \to \infty$ is
\[\sqrt{\pi } \sqrt{n}-\log (n)+\frac{1}{2} (5-6 \gamma ) +O(n^{-1/2}).\]
\end{theorem}

\begin{remark} The asymptotic formula of Theorem~\ref{Th3} when $n=200$ yields 20.536 for the average capacity. Using the exact formula of Theorem~\ref{Th4} divided by the Catalan number for $n=200$ yields 20.368. Taking larger $n$ improves the accuracy.
\end{remark}

\end{document}